\documentclass{PM}

\usepackage{tikz}
\usetikzlibrary{positioning}
\usetikzlibrary{matrix}

\newtheorem{theorem}{Theorem}[section]

\newtheorem{corollary}[theorem]{Corollary}
\newtheorem{proposition}[theorem]{Proposition}

\theoremstyle{definition}
\newtheorem{remark}[theorem]{Remark}

\newcommand{\trop}{\operatorname{trop}}

\newcommand{\an}{\operatorname{an}}

\newcommand{\girth}{\mathrm{girth}}

\received{}
\title{The Locus of Brill-Noether General Graphs is not Dense}
\address[dave.jensen@uky.edu]{David Jensen, University of Kentucky, Lexington, KY}
\author{David Jensen}
\date{}
\begin{document}

\maketitle

\begin{abstract}
We provide an example of a trivalent, 3-vertex connected graph $G$ such that, for any choice of metric on $G$, the resulting metric graph is Brill-Noether special.
\end{abstract}

\begin{classification}
Primary 14T05; Secondary 14H51.
\end{classification}

\begin{keywords}
Chip-firing, Brill-Noether theory, tropical curves.
\end{keywords}

\section{Introduction}

We say that an algebraic curve $C$ is Brill-Noether general if, for all positive integers $r$ and $d$, the variety $W^r_d (C)$ parameterizing divisors of degree $d$ and rank at least $r$ has dimension equal to the Brill-Noether number $\rho (g,r,d) = g-(r+1)(g-d+r)$, and is empty when $\rho$ is negative.  Otherwise, we say that $C$ is Brill-Noether special.  By the Brill-Noether Theorem \cite{GriffithsHarris80}, the locus of Brill-Noether general curves is a dense open subset of $M_g$.  The Baker-Norine theory of divisors on metric graphs gives us an analogous notion of Brill-Noether general graphs \cite{Baker08,BakerNorine07}.  As in the classical case, the locus of Brill-Noether general graphs in the moduli space of tropical curves $M_g^{\trop}$ is open \cite{LPP12,Len12} and non-empty \cite{tropicalBN}, but this does not imply that it is dense. Specifically, $M_g^{\trop}$ is stratified by the sets $M_G^{\trop}$ consisting of all metric graphs with the same underlying discrete graph $G$ \cite{BrannettiMeloViviani11}, and the question of which strata contain Brill-Noether general curves remains an open problem.

The top-dimensional strata of $M_g^{\trop}$ correspond to trivalent graphs, and it is a straightforward exercise to construct a trivalent graph $G$ with the property that every metric graph $\Gamma \in M_G^{\trop}$ is Brill-Noether special.  For example, if $G$ is the graph pictured in Figure \ref{Fig:Cheating}, obtained by attaching a loop to each leaf of a tree, then every $\Gamma \in M_G^{\trop}$ is hyperelliptic.  Prior to this note, however, all known examples of such graphs contained bridges.  This is a bit unsatisfying, as the length of the bridges does not affect either the Jacobian or the Brill-Noether theory of the graph, and for this reason it is customary to treat graphs without bridges as the proper analogues of algebraic curves (see, for example, \cite[Remark 4.8]{BakerNorine07}).  More precisely, one might ask if the locus of Brill-Noether general graphs has dense image in the moduli space of tropical Jacobians.  In \cite[Theorem 4.1.9]{CaporasoViviani10}, it is shown that the Torelli theorem holds for 3-vertex connected metric graphs -- that is, the Jacobians of two 3-vertex connected metric graphs are isomorphic as principally polarized tropical abelian varieties if and only if the two graphs are isomorphic as tropical curves\footnote{We note that, for trivalent graphs, the property of 3-vertex connectivity is equivalent to that of 3-egde connectivity (see, for example, \cite[Lemma A.1.2]{CaporasoViviani10}).}.  It is therefore more natural to ask for a 3-vertex connected trivalent graph $G$ such that every $\Gamma \in M_G^{\trop}$ is Brill-Noether special.  The question of whether such graphs exist has appeared in several places, for example in \cite[p. 6]{LPP12}.  In this note we provide an example of such a graph.

\begin{figure}
\centering

\begin{tikzpicture}

\draw (-1.5,0) circle (0.5);
\draw (1.2,.7) circle (0.5);
\draw (1.2,-.7) circle (0.5);
\draw (-1,0)--(0,0);
\draw (0,0)--(.7,.7);
\draw (0,0)--(.7,-.7);
\draw [ball color=black] (-1,0) circle (0.55mm);
\draw [ball color=black] (0,0) circle (0.55mm);
\draw [ball color=black] (.7,.7) circle (0.55mm);
\draw [ball color=black] (.7,-.7) circle (0.55mm);

\end{tikzpicture}
\caption{A graph that is hyperelliptic for any choice of edge lengths}
\label{Fig:Cheating}
\end{figure}
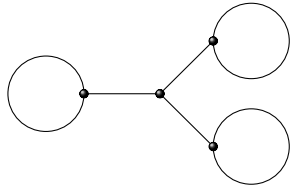

Our example is the Heawood graph, which is the Levi graph of the Fano plane.  This graph, depicted in Figure \ref{Fig:Heawood}, has 14 vertices, corresponding to the 7 points and 7 lines in the Fano plane, with an edge between two vertices if the corresponding point lies on the corresponding line.  Among the many fascinating combinatorial properties of the Heawood graph is the fact that it is the smallest trivalent graph with girth 6, and that it is the unique symmetric graph of genus 8.  Our main result is the following.

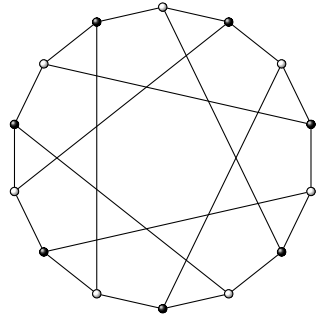
\begin{figure}
\centering

\begin{tikzpicture}

\draw (1.9498,.445)--(1.5636,1.247);
\draw (1.5636,1.247)--(.8678,1.802);
\draw (.8678,1.802)--(0,2);
\draw (0,2)--(-.8678,1.802);
\draw (-.8678,1.802)--(-1.5636,1.247);
\draw (-1.5636,1.247)--(-1.9498,.445);
\draw (-1.9498,.445)--(-1.9498,-.445);
\draw (-1.9498,-.445)--(-1.5636,-1.247);
\draw (-1.5636,-1.247)--(-.8678,-1.802);
\draw (-.8678,-1.802)--(0,-2);
\draw (0,-2)--(.8678,-1.802);
\draw (.8678,-1.802)--(1.5636,-1.247);
\draw (1.5636,-1.247)--(1.9498,-.445);
\draw (1.9498,-.445)--(1.9498,.445);

\draw (1.9498,.445)--(-1.5636,1.247);
\draw (1.5636,1.247)--(0,-2);
\draw (.8678,1.802)--(-1.9498,-.445);
\draw (0,2)--(1.5636,-1.247);
\draw (-.8678,1.802)--(-.8678,-1.802);
\draw (1.9498,-.445)--(-1.5636,-1.247);
\draw (-1.9498,.445)--(.8678,-1.802);
\draw [ball color=black] (1.9498,.445) circle (0.55mm);
\draw [ball color=white] (1.5636,1.247) circle (0.55mm);
\draw [ball color=black] (.8678,1.802) circle (0.55mm);
\draw [ball color=white] (0,2) circle (0.55mm);
\draw [ball color=black] (-.8678,1.802) circle (0.55mm);
\draw [ball color=white] (-1.5636,1.247) circle (0.55mm);
\draw [ball color=black] (-1.9498,.445) circle (0.55mm);
\draw [ball color=white] (-1.9498,-.445) circle (0.55mm);
\draw [ball color=black] (-1.5636,-1.247) circle (0.55mm);
\draw [ball color=white] (-.8678,-1.802) circle (0.55mm);
\draw [ball color=black] (0,-2) circle (0.55mm);
\draw [ball color=white] (.8678,-1.802) circle (0.55mm);
\draw [ball color=black] (1.5636,-1.247) circle (0.55mm);
\draw [ball color=white] (1.9498,-.445) circle (0.55mm);

\end{tikzpicture}
\caption{The Heawood graph}
\label{Fig:Heawood}
\end{figure}

\begin{theorem}
\label{Thm:Heawood}
If $G$ is the Heawood graph, then any metric graph $\Gamma \in M_G^{\trop}$ possesses a divisor of degree 7 and rank 2.  Since $g(G) = 8$ and $\rho (8,2,7) = -1$, every such metric graph $\Gamma$ is Brill-Noether special.
\end{theorem}

The Fano plane is an example of a rank 3 matroid, and in \cite{Cartwright14} such matroids are studied in the context of the divisor lifting problem.  More specifically, given a metric graph $\Gamma$ and a divisor $D$ on $\Gamma$, one can ask whether there exists an algebraic curve and a divisor on the curve of the same rank as $D$ specializing to $\Gamma$ and $D$ respectively.  Such a pair of a curve and a divisor is called a \emph{lifting} of the pair $(\Gamma,D)$.  Among the results of \cite{Cartwright14} is the fact that, if $\Gamma$ is the Heawood graph with all edges of length one, and $D$ is the divisor of degree 7 and rank 2 described in Theorem \ref{Thm:Heawood}, then the pair $(\Gamma,D)$ admits a lifting over a valued field $K$ if and only if the characteristic of $K$ is 2.  One consequence of Theorem \ref{Thm:Heawood} is the following, which is valid in any characteristic.

\begin{corollary}
\label{Cor:Lifting}
Let $G$ be the Heawood graph and let $\Gamma \in M_G^{\trop}$ have generic edge lengths.  Then there exists a divisor $D$ on $\Gamma$ of degree 7 and rank 2 such that the pair $(\Gamma,D)$ does not admit a lifting.
\end{corollary}

\medskip

\noindent \textbf{Acknowledgments.}  This paper was written during a conference on specialization of divisors from curves to graphs at Banff International Research Station, in response to a question posed independently by several different participants.  We would like to thank Matt Baker, Dustin Cartwright, Ethan Cotterill, Yoav Len and Sam Payne for helpful conversations.  We would also like to thank the referees for several suggestions that helped improve the paper.  The author was supported in part by an AMS Simons Travel Grant.

\section{The Example}\label{Sec:Example}

\begin{proposition}
\label{Prop:RankBound}
Let $G$ be a graph, $B \subset V(G)$ a subset of the vertices of $G$, and $D_B = \sum_{v \in B} v$.  If every cycle of $G$ contains at least $r+1$ vertices in $B$, then for any metric graph $\Gamma \in M_G^{\trop}$, we have $r_{\Gamma} (D_B) \geq r$.
\end{proposition}

\begin{proof}
We prove this by induction on $r$, the case $r=0$ being obvious.  We first show, following \cite[Example 3.10]{Luo11}, that if $r \geq 1$, then $B$ is a rank-determining set.  By assumption, every cycle in $G$ contains at least 2 vertices in $B$.  It follows that, if $U$ is a connected component of $\Gamma \smallsetminus B$, then its closure $\overline{U}$ does not contain a cycle.  By \cite[Proposition 3.9]{Luo11}, we therefore have $\overline{U} \subseteq \mathcal{L} (B)$.  Since the sets $\overline{U}$ cover $\Gamma$, this implies that $\mathcal{L} (B) = \Gamma$, or $B$ is a rank-determining set.  From this we see that $r_{\Gamma} (D_B) \geq r$ if and only if $r_{\Gamma} (D_B - v) \geq r-1$ for every $v \in B$.  By assumption, however, every cycle in $G$ contains at least $r$ vertices in $B \smallsetminus \{ v \}$, so by our inductive hypothesis, $r_{\Gamma} (D_B - v) \geq r-1$.
\end{proof}

Recall that the girth of a graph is the minimum number of vertices in a cycle.

\begin{corollary}
\label{Cor:Bipartite}
Let $G$ be a bipartite graph, and let $B$ denote the set of vertices of one color.  Then, for any metric graph $\Gamma \in M_G^{\trop}$, we have
$$ r_{\Gamma} (D_B) \geq \frac{1}{2} \girth (G) - 1 . $$
\end{corollary}

\begin{remark}
In general, the divisor $D_B$ may have much larger rank than the bound given in Corollary \ref{Cor:Bipartite}.  For example, given any trivalent graph $G$, we may construct a bipartite graph $G'$ by introducing a vertex in the middle of each edge of $G$.  If $B$ is the set of original vertices of the graph $G$, then $D_B$ is the canonical divisor on $G'$, which is known to have rank $g(G)-1$.  Since there exist graphs of arbitrarily high genus with girth 1, we may construct divisors of the form $D_B$ for which the bound is arbitrarily bad.
\end{remark}

\begin{proof}[Proof of Theorem~\ref{Thm:Heawood}]
The Heawood graph is a bipartite graph of girth 6, so by Corollary \ref{Cor:Bipartite} the divisor $D_B$ has rank at least 2.  To see that the rank is exactly 2, choose any two vertices $v_1 \neq v_2 \in B$ and note that by Dhar's burning algorithm \cite{Dhar90} \cite[Algorithm 2.5]{Luo11}, the divisor $D_B - v_1 - v_2$ is $v_1$-reduced.
\end{proof}

\begin{remark}
Corollary \ref{Cor:Bipartite} does not yield any other example of a trivalent, bipartite graph $G$ such that every metric graph $\Gamma \in M_G^{\trop}$ is Brill-Noether special. Indeed, the corollary produces a divisor of degree $d = g-1$ and rank at least $r = \frac{1}{2} \girth (G) - 1$, for which the Brill-Noether number is
$$ \rho (g,r,d) = g - (r+1)(g-d+r) = g - \frac{1}{4} \girth (G)^2 . $$
We will show that the Heawood graph is the only trivalent graph satisfying the inequality
$$ \frac{1}{4} \girth (G)^2 > g(G) . $$
To see this, note that if $G$ is a trivalent graph of given girth, then we may obtain a lower bound on the number of vertices $n$ by performing a breadth first search starting from any vertex.  This procedure yields the bound
$$ n \geq 2 ( 2^{\frac{1}{2}\girth (G)} - 1) $$
(see, for example, \cite{ExooJajcay08}).  Since $G$ is trivalent, we have
$$ g(G) = \frac{n}{2} + 1 \geq 2^{\frac{1}{2}\girth (G)} , $$
so the inequality above is satisfied if and only if $G$ has genus 8 and girth 6.
\end{remark}

\begin{proof}[Proof of Corollary~\ref{Cor:Lifting}]
Let $\overline{BN}^{2,\an}_7$ be the analytification of the the Brill-Noether locus $\overline{BN}^2_7$ inside of $\overline{M}_8^{\an}$.  By the Brill-Noether Theorem, since $\rho (8,2,7) = -1$, we see that the general point of $\overline{M}_8$ is not contained in $\overline{BN}^2_7$.  Since $\overline{M}_8$ is irreducible of dimension 21, the closed subvariety $\overline{BN}^2_7$ is at most 20-dimensional.  It follows that the image of $\overline{BN}^{2,\an}_7$ under the retraction $\overline{M}_8^{\an} \to \overline{M}_8^{\trop}$ of \cite{acp} has dimension at most 20.  If $G$ is the Heawood graph, however, then $M_G^{\trop}$ is 21-dimensional, so the general point $\Gamma \in M_G^{\trop}$ is not contained in the image of $\overline{BN}^{2,\an}_7$.  It follows that if $\Gamma$ is such a general point, and $C$ is any curve such that the skeleton of $C^{\an}$ is isometric to $\Gamma$, then $C$ is Brill-Noether general, and hence every divisor $D$ on $C$ that specializes to $D_B$ has rank less than 2.
\end{proof}

\bibliography{math}

\end{document}